\documentclass[12pt]{amsart}
\usepackage{amsmath,amssymb,amsthm}
\usepackage{color}

\newtheorem{lem}{Lemma}[section]
\newtheorem{prop}{Proposition}[section]
\newtheorem{cor}{Corollary}[section]

\newtheorem{thm}{Theorem}[section]

\theoremstyle{definition}

\theoremstyle{remark}

\theoremstyle{remark}
\newtheorem{remark}{Remark}[section]
\numberwithin{equation}{section}

\newcommand{\set}[1]{\left\{#1\right\}}
\newcommand{\C}{{\mathbb C}}
\newcommand{\N}{{\mathbb N}}

\newcommand{\R}{{\mathbb R}}

\begin{document}

\title[Stable standing waves  for a class of nonlinear SP equations]{Stable standing waves  for a class of nonlinear  Schr\"odinger-Poisson equations}

\author{Jacopo Bellazzini}
\address{J. Bellazzini, \newline Dipartimento di Matematica Applicata ``U. Dini'', University of Pisa, via Buonarroti 1/c, 56127 Pisa, ITALY}%
\email{j.bellazzini@ing.unipi.it}%
\author{Gaetano Siciliano}
\address{G. Siciliano, \newline Dipartimento di Matematica, Universit\`a degli Studi di Bari, via Orabona 4, 70125 Bari, Italy}%
\email{siciliano@dm.uniba.it}
\thanks{The authors are partially supported by M.I.U.R project PRIN2007 ``Variational and topological methods in the study of nonlinear phenomena'' }
\thanks{The second author is also supported by J. Andaluc\'ia (FQM  116)}

\maketitle

\begin{abstract}
We prove the existence of orbitally stable standing waves with prescribed 
$L^2$-norm for the following  Schr\"odinger-Poisson type equation
\begin{equation*}\label{intro}
i\psi_{t}+ \Delta \psi - (|x|^{-1}*|\psi|^{2}) \psi+|\psi|^{p-2}\psi=0 \ \ \  \text{ in } \R^{3},\\
\end{equation*} 
when $p\in \left\{ \frac{8}{3}\right\}\cup (3,\frac{10}{3})$.  In the case $3<p<\frac{10}{3}$ we prove the existence  and stability only for sufficiently large $L^2$-norm. 
In case $p=\frac{8}{3}$  our approach recovers  the result of Sanchez and Soler \cite{SS} 
for sufficiently small charges. The main point is the analysis  of the compactness of minimizing sequences
for the related constrained minimization problem. In a final section a further application to the Schr\"odinger equation involving the biharmonic operator is given.
\end{abstract}

\section{Introduction}
In this paper we study the following  Schr\"odinger-Poisson type equation
\begin{equation}\label{SP}
i\psi_{t}+ \Delta \psi - (|x|^{-1}*|\psi|^{2}) \psi+|\psi|^{p-2}\psi=0 \ \ \  \text{ in } \R^{3},\\
\end{equation} 
where $\psi(x,t):\R^{3}\times[0,T)\rightarrow \C$ is the wave function, 
$*$ denotes the convolution  and $2<p<{10}/{3}$. It is known  that in this case  
 the Cauchy problem associated to \eqref{SP} is globally well-posed in $H^{1}(\R^{3};\C)$ (see e.g. \cite{C}).

 We are interested in the search of standing wave solutions of \eqref{SP}, namely solutions of the form
 $$\psi(x,t)=e^{-i\omega t}u(x) \,,\ \ \ \omega\in \R,\ \ u(x)\in \C\,,$$
 so we are reduced to study the following semilinear elliptic equation with a non local nonlinearity
\begin{equation}\label{eq}
- \Delta u + \phi_{u} u-|u|^{p-2}u=\omega u \ \ \  \text{ in } \R^{3}
\end{equation}
where we have set
 $$\phi_{u}(x)=\int_{\R^{3}}\frac{|u(y)|^{2}}{|x-y|}dy\,.$$
 Evidently, $\phi_{u}$ satisfies $-\Delta \phi_{u}=4\pi |u|^{2}$, is uniquely determined by $u$
 and is usually interpreted as the scalar potential of the 
 electrostatic field generated by the charge density $|u|^{2}$. 
 
Because of its importance in many different physical framework, many authors have investigated
 the Schr{\"o}dinger-Poisson system  (sometimes called Schr{\"o}dinger-Poisson-Slater system).
Besides the paper of Benci and Fortunato  \cite{BF} on a bounded domain,
many papers on ${\R}^{3}$ have treated different
aspects of this system, even with an additional external and fixed potential $V(x)$.
In particular ground states, radially and non-radially solutions are studied, see e.g.
 \cite{AzzPo, Coc, DApMu, dAv, Kik, jfa, WZh}. However in all this papers the frequency $\omega$
is seen as a parameter so the authors 
deal with the functional
$$\frac{1}{2}\int_{\R^{3}}|\nabla u|^{2}dx+\frac{\omega}{2}\int_{\R^{3}}u^{2}dx+\frac{1}{4}\int_{\R^{3}}\phi_{u} u^{2}dx-\frac{1}{p}\int_{\R^{3}}|u|^{p}dx$$
and look for its critical points in $H^{1}(\R^{3};\R).$ In this approach nothing can be said a priori on the $L^{2}$-norm of the solution.
On the other hand in \cite{PS} the problem has been studied in a bounded domain $\Omega$ with a nonhomogeneous Neumann boundary condition on the potential $\phi_{u}$:
here the compatibility condition for $\phi_{u}$ imposes to study a constrained problem on $\{u\in H_{0}^{1}(\Omega): \|u\|_{2}=1\}.$ 


In spite of the above cited papers on $\R^3$, we look for solutions $u$ with a priori prescribed  $L^{2}$-norm. 
The natural way to study the problem is to look for the constrained critical points of the functional 
 $$I(u)=\frac{1}{2}\int_{\R^{3}}|\nabla u|^{2}dx+\frac{1}{4}\int_{\R^{3}}\phi_{u}| u|^{2}dx-\frac{1}{p}\int_{\R^{3}}|u|^{p}dx$$
  on the $L^{2}$-spheres in $H^{1}(\R^{3};\C)$
 $$B_{\rho}=\{u\in H^{1}(\R^{3};\C): \|u\|_{2}=\rho\}.$$
 So by a solution of \eqref{eq} we mean  a couple 
 $(\omega_{\rho}, u_\rho)\in\R\times H^{1}(\R^{3};\C)$ where $\omega_{\rho}$ is the Lagrange multiplier associated to the critical point $u_{\rho}$ on $B_{\rho}.$
%
%

Actually we are interested in the existence of solutions of \eqref{eq} with minimal energy (constrained to the sphere), i.e. to the minimization problem
\begin{equation}\label{minim}
I_{\rho}=\inf_{
B_{\rho}} I(u)
\end{equation}
that makes sense for $2<p<10/3$; indeed it is well known that in this case the $C^{1}$ functional
$I$ is bounded from below and coercive
on $B_{\rho}$
(see Lemma \ref{stimap}). 
As far as we know the only results on constrained minimization for  nonlinear Schr\"odinger-Poisson are \cite{SS} in case $p=\frac{8}{3}$ and \cite{K} for $p=3$. In \cite{SS}
the authors prove that all the minimizing sequence for \eqref{minim} are compact
provided that $\rho$ is sufficently small. In \cite{K} the author proves that if $\Lambda$ is sufficently large then the infimum of the minimization problem
\begin{equation*}
I_{\rho}=\inf_{
B_{\rho}} I_{\Lambda}(u)
\end{equation*}
where
$$I_{\Lambda}(u):=\frac{1}{2}\int_{\R^{3}}|\nabla u|^{2}dx+\frac{1}{4}\int_{\R^{3}}\phi_{u}| u|^{2}dx-\frac{\Lambda}{3}\int_{\R^{3}}|u|^{3}dx$$
is achieved for any $\rho$.\\
It is known that, in this kind of problems, that main difficulty concerns with the lack of compactness  of the (bounded) minimizing sequences $\{u_n\}\subset B_{\rho}$; indeed  two possible bad scenarios are possible:
\begin{itemize}
\item $u_n\rightharpoonup 0$;
\item $u_n \rightharpoonup \bar u\neq 0$ and $0<\|\bar u\|_2<\rho$.
\end{itemize}
In order to avoid the above two cases and to show that 
the infimum is achieved,
we prove a lemma (Lemma \ref{main-abs})  in  an abstract framework that guarantees the compactness
of the minimizing sequences in the right norm. We recall that the abstract lemma is essentially contained in \cite{BB} and here it has been modified for the application to a wider class of functionals. 
Roughly  speaking, this  lemma is a version of the Concentration Compactness principle of \cite{L} having in mind
the application to a constrained minimization problem for  functionals of the form $$I(u)=\frac{1}{2}\|u\|^{2}_{D^{m,2}}+T(u).$$
 The lemma we prove says that  if  $\bar u\neq 0$ and 
$T(u)$ has a \emph{splitting}
property, i.e 
$$ T(u_n-\bar u)+T(\bar u) =T(u_n)+o(1)$$
and the infima are \emph{subadditive} in the following sense
$$ 
 I_{\rho}<I_{\mu}+I_{\sqrt{\rho^2-\mu^2}} \ \  \text{ for any }\ \  0<\mu<\rho\,,
$$
then
$\|u_n -\bar u\|_{H^{m}}=o(1)$ 
and, as a consequence,  $\|\bar u\|_2=\rho.$ 

As a consequence of the abstract minimization lemma we prove  the following



\begin{thm}\label{MT1}
Let  $p\in \left\{ \frac{8}{3} \right\} \cup (3, \frac{10}{3})$.  Then there exist $\rho_1>0$ and $\rho_2>0$ (depending on $p$) such that all the minimizing sequences for \eqref{minim} are precompact in $H^1(\R^3;\C)$ up to translations provided that
\begin{eqnarray}
&&0<\rho<\rho_1 \ \text{ if } p=\frac{8}{3} \nonumber\\
&&\rho_2<\rho<+\infty \ \text{ if } 3<p<\frac{10}{3} \nonumber.
\end{eqnarray}
In particular there exists a couple $( \omega_{\rho}, u_{\rho})\in\R\times H^{1}(\R^{3};\R)$  solution of \eqref{eq}.
\end{thm}
\begin{remark}
We underline that the result for $p=\frac{8}{3}$ it has been proved first 
by \cite{SS}
with a different approach  to that developed in this paper. However it is interesting that our result for $p=\frac{8}{3}$ is proved within the same general framework that is applied for $3<p<\frac{10}{3}$.
\end{remark}
As a matter of fact  there are few results concerning the orbital stability of 
standing waves for Schr\"odinger-Poisson equation. We mention 
\cite{IL} and \cite{K} where the orbital stability is achieved by following
the original approach of \cite{GSS}. On the other hand, following \cite{CL} and
\cite{SS},
the compactness of minimizers on $H^1(\R^3;\C)$ and the conservation laws 
give rise to the orbital stability of the standing waves $\psi =e^{-i\omega_{\rho} t}u_{\rho}$ 
without further efforts; so we get the following

\begin{thm}\label{stabil}
Let  $p\in \left\{\frac{8}{3} \right\}\cup (3, \frac{10}{3})$.
Then 
the set 
$$S_{\rho}=\{e^{i\theta }
u(x): \theta\in [0,2\pi), \|u\|_{2}=\rho, \ 
 I(u)=I_{\rho}\}$$
is orbitally stable.

\end{thm}

The definition of {\sl orbital stability} is recalled in the Section 4.

\medskip

We underline that Lemma \ref{main-abs} can be applied to
a wider class of minimization problems involving, for instance the biharmonic operator. 
For this reason, in the final Section 5 we study the following minimization problem
$$
J_{\rho}=\inf_{ B_\rho} 
\left (\frac 12 \|\Delta u\|_2^2 + \int_{\R^N}F(u)dx \right)
$$
where $B_{\rho}=\{u\in H^{2}(\R^{N}): \|u\|_{2}=\rho\}$ and the nonlinear local
term $F:H^2(\R^N)\rightarrow \R$ fulfills some suitable assumptions that  will be specified later. As a byproduct we obtain the orbital stability 
for the standing waves of the following Schr\"odinger equation involving the bilaplace operator
$$
i\psi_{t} - \Delta^2 \psi  -F'(|\psi|) \frac{\psi}{|\psi|}=0,
\hbox{ } (x,t) \in \R^N\times\R.
$$

\subsection{ Notation}  In all the paper it is understood the all the functions,
unless otherwise stated, are
complex-valued, but we will write simply $L^{s}(\R^{3}), H^{1}(\R^{3})....$ where, for any $1\leq s < +\infty$, $L^s(\R^3)$ is the usual Lebesgue space endowed
with the norm
$$||u||_{s}^s:=\int_{\R^{3}} |u|^sdx,$$
and
$H^1(\R^3)$ the usual Sobolev space  endowed with the norm
$$\|u\|_{H^1}^2:=\int_{\R^{3}} |\nabla u|^2dx+\int_{\R^{3}} |u|^2dx.$$
In order to state the abstract  lemma  
let us the space $D^{m,2}(\R^N)$. It is defined 
 as the completion of $C_0^{\infty}(\R^N)$ with respect to the norm
$$\|u\|_{D^{m,2}}^2:=\sum_{\alpha_{1}+...+\alpha_{N}=m} \int_{\R^N}|D^{\alpha}u|^2dx \ \ \text{where} \ \ 
\alpha\in \N^{N},
D^{\alpha}=\partial_{x_{1}}^{\alpha_{1}}\cdot\cdot\, \partial_{x_{N}}^{\alpha_{N}}$$
We need also $H^m(\R^N)$, the usual Sobolev space with norm
$$||u||_{H^m}^2:=\|u\|_{D^{m,2}}^2+\|u\|_2^2.$$
We will use $C$ to denote a suitable positive constant whose value may change also in the same line
and the symbol $o(1)$ to denote a quantity which goes to zero. We also use $O(1)$ to denote a bounded sequence.

\bigskip

The paper is organized as follows: Section 2 is devoted to the minimization
problem and to the proof of the abstract lemma. Section 3 concerns the proof
of the main theorem while in Section 4 the orbital stability of the standing waves is proved.
In the final Section 5 the abstract lemma is applied to the biharmonic Schr\"odinger equation.

\medskip

\section{The minimization problem}
As we have anticipated, we first prove an abstract result  on  a constrained minimization problem on  Sobolev  spaces $H^m(\R^N), N\ge3$. 
Let we consider the following  problem
\begin{equation*} \label{mini1}
I_{\rho}=\inf_{B_{\rho}} \ I(u)
\end{equation*}
where $ B_{\rho}:=\{ u \in H^m(\R^N) \text{ such that } \|u\|_{2}=\rho\}$ and
\begin{equation} \label{func1}
I(u):= \frac 12   \|u \|_{D^{m,2}}^2 +T(u)
\end{equation}
Under suitable assumption on $T$ we have the strong convergence of the weakly
convergent minimizing sequence.

\begin{lem} \label{main-abs}
Let $T$ be a $C^{1}$ functional on $H^m(\R^{N})$ and $\set{u_n}\subset  B_{\rho}$ be a minimizing sequence for $I_{\rho}$
such that $u_n\rightharpoonup \bar u\neq 0$; let us set $\mu= \|\bar u\|_{2}\in(0,\rho]$. 

Assume also that
 \begin{equation}\label{(2)}
T(u_n-\bar u) + T(\bar u)=T(u_n)+ o(1);
\end{equation}

\begin{equation}\label{(3)}
T(\alpha_{n}(u_{n}-\bar u))-T( u_n-\bar u)=o(1)
\end{equation}
where $\alpha_{n}= {\sqrt{\rho^2-\mu^2}}/{\| u_n-\bar u\|_2}$
%
and finally that
\begin{equation}\label{(4)}
 I_{\rho}<I_{\mu}+I_{\sqrt{\rho^2-\mu^2}} \ \  \text{ for any }\ \  0<\mu<\rho\,.
\end{equation}
Then $\bar u \in B_{\rho}$. 

Moreover if, as $n,m \to +\infty$
\begin{equation}\label{(5)}
<T'(u_{n}) - T'(u_{m}), u_{n} - u_{m}> = o(1)
\end{equation}
%
%
\begin{equation}\label{(6)}
<T'(u_{n}), u_{n}> = O(1)
\end{equation}
%
then $||u_n-\bar u||_{H^m(\R^{N})}\rightarrow 0$.
\end{lem}

\begin{proof}
We argue by contradiction and assume that $\mu_{0} < \rho$. Since $u_{n}-\bar u\rightharpoonup0$,
$$
\| u_n-\bar u\|_2^2+\|\bar u\|_2^2=\| u_n\|_2^2+o(1)
$$
hence 
\begin{equation}\label{a}
\alpha_{n}=\frac {\sqrt{\rho^2-\mu^2}}{\| u_n-\bar u\|_2}\rightarrow 1.
\end{equation}
Since $\{u_{n}\}$ is a minimizing sequence, we get
$$
\frac 12 \| u_{n} \|_{D^{m,2}}^2 +T(u_n)=I_{\rho}+o(1)
$$
and by (\ref{(2)}), we deduce also
$$
\frac 12 \| u_n-\bar u\|_{D^{m,2}}^2 +\frac 12 \|\bar u\|_{D^{m,2}}^2+T(u_n-\bar u)+T(\bar u)=I_{\rho}+o(1).
$$
Hence using \eqref{a} and (\ref{(3)}) we infer 
$$
\frac 12 \| \alpha_{n}(u_n-\bar u) \|_{D^{m,2}}^2 +\frac 12 \|\bar u\|_{D^{m,2}}^2 +T(\alpha_{n}(u_n-\bar u))+T(\bar u)=I_{\rho}+o(1).
$$
%
Finally, notice that $\| \alpha_{n}(u_n-\bar u) \|_2 = \sqrt{\rho^{2}-\mu^{2}}$, therefore 
$$
I_{_{\sqrt{\rho^{2}-\mu^{2}}}} + I_{\mu} \leq  I_{\rho} + o(1)
$$
which is in contradiction with \eqref{(4)}. This implies that $\| \bar u \|_2= \rho$.

To prove the second assertion, we may assume, by the Ekeland variational principle, that $\set{u_n}$ is a Palais-Smale sequence
for the functional $I$. From $\bar u \in B_{\rho}$ it follows that $\| u_n-\bar u
\|_2 = o(1)$, hence it remains to show that $\| u_{n} - \bar
u \|_{D^{m,2}} = o(1)$ up to a sub-sequence.
By assumptions there exists a sequence $\set{\lambda_{n}}\subset
\R$ such that for the functional $I$ defined in (\ref{func1})
$$
<I'(u_{n}) - \lambda_{n}\, u_{n} , v > = o(1) \qquad \forall\, v \in H^{m}(\R^{N})
$$
where $<\cdot,\cdot>$ denotes the duality pairing. It follows that
$$
<I'(u_{n}) - \lambda_{n}\, u_{n} , u_{n} > = o(1)
$$
since $\| u_{n} \|_{H^m}$ is bounded. From this and assumption (\ref{(6)}) it follows that the sequence $\set{\lambda_{n}}$ is bounded, hence up to a sub-sequence there exists $\lambda \in \R$ with $\lambda_{n} \to \lambda$.

We now have
$$
<I'(u_{n}) - I'(u_{m}) - \lambda_{n}u_{n} + \lambda_{m} u_{m}\ ,\ u_{n} - u_{m}>  = o(1) \qquad \text{as}\ n,m\to \infty
$$
hence, using that $(\lambda_{n} - \lambda_{m}) <u_{m}, u_{n} - u_{m}> = o(1)$,
$$
\| u_{n} - u_{m} \|_{D^{m,2}}^{2} + <T'(u_{n}) - T'(u_{m}), u_{n} - u_{m}>
 - \lambda_{n} \| u_{n} - u_{m}\|_{2}^{2} = o(1)
$$
Since $\| u_{n} - u_{m}\|_2 = o(1)$, $\lambda_{n} \to
\lambda$ and (\ref{(5)}) holds, we obtain that $\set{u_{n}}$ is a
Cauchy sequence in $H^m(\R^N)$. Hence $\|u_{n} - \bar u\|_{H^m} \to 0$.
\end{proof}

\section{Proof of the Main Theorem}
We want to apply the previous theorem to the functional $I:H^1\rightarrow \R$
given by
\begin{equation*}
I(u)=\frac{1}{2}\int_{R^{3}} |\nabla u|^{2}dx+\frac{1}{4}\int_{\R^{3}} \phi_u |u|^2dx-\frac{1}{p}\int_{\R^{3}} |u|^pdx\,.
\end{equation*}
with 
$$T(u):=N(u)+M(u)$$
where
$$N(u)=\frac{1}{4}\int_{\R^{3}}\phi_{u }|u|^{2} dx ,\ \ \ M(u)=-\frac{1}{p}\int_{\R^{3}}|u|^{p}dx.$$

Before to prove the main theorem some preliminaries are in order: the next lemma shows that the the functional is bounded from below on $B_{\rho}.$ 
\begin{lem}\label{stimap}
If $2<p<\frac{10}{3}$, then for every $\rho>0$ the functional $I$ is bounded from below and coercive on $B_\rho$.
\end{lem}
\begin{proof}
We apply the following Sobolev inequality 
\begin{equation*}\label{Sobolev}
||u||_q\leq b_q||u||_2^{1-\frac{N}{2}+\frac{N}{q}}||\nabla u||_2^{\frac{N}{2}-\frac{N}{q}}
\end{equation*}
that holds for $2 \leq q \leq 2^{*}$ when $N \geq 3$. Therefore
if
$||u||_{2}=\rho$ it follows
$||u||_p^p \leq b_{p,\rho} ||\nabla u||_2^{\frac{3p}{2}-3}$
and
\begin{eqnarray*}
  I(u)&\geq & \int \left( \frac 12 |\nabla u|^2 -\frac{1}{p}|u|^{p}\right)dx\\
&\geq &\frac 12\|\nabla u\|^2_{2}-
b_{p,\rho}\|\nabla u\|^{\frac{3p}{2}-3}_{2}
\end{eqnarray*}
Since $p < \frac{10}{3}$, it results $\frac{3p}{2}-3<2$ and
\begin{equation*}
 I(u)\geq \frac{1}{2}||\nabla u||_2^2 +O(||\nabla u ||_2^2).
\end{equation*}
which concludes the proof.
\end{proof}
Notice that if we set   $u_{\lambda}(\cdot)=\lambda^{\alpha}u(\lambda^{\beta}(\cdot))$,
 $\alpha,\beta\in \R,\lambda>0$, then

\begin{equation*}
\phi_{u_{\lambda}}(x)
=\int_{\R^{3}}\frac{\lambda^{2\alpha+\beta}|u(\lambda^{\beta}y)|^{2}}{|\lambda^{\beta}x-\lambda^{\beta}y|}dy
=\lambda^{2(\alpha-\beta)}\int_{\R^{3}}\frac{|u(y)|^{2}}{|\lambda^{\beta}x-y|}dy=\lambda^{2(\alpha-\beta)}\phi_{u} (\lambda^{\beta}x).
\end{equation*}
%

Now we prove some subadditivity properties that are  crucial for the proof of Theorem \ref{MT1}. 
\begin{lem}\label{fondamentale}
Let  $p=\frac{8}{3}$, then here exists $\rho_{1}>0$ such that $I_{\mu}<0$ for all $\mu \in (0, \rho_1)$ and  
\begin{equation*}
I_{\rho}<I_{\mu}+I_{\sqrt{\rho^2-\mu^2}}
\end{equation*}
for all 
 $0<\mu<\rho<\rho_{1}$.\\
If $3<p<\frac{10}{3}$, then  there exists $\rho_{2}>0$ such that $I_{\mu}<0$ for all $\mu \in (\rho_2, +\infty)$
and 
\begin{equation*}
I_{\rho}<I_{\mu}+I_{\sqrt{\rho^2-\mu^2}}
\end{equation*}
for all $\rho>\rho_2$ and $0<\mu<\rho$.
\end{lem}
\begin{proof}
We define $u_{\theta}(x)=\theta^{1-\frac{3}{2}\beta}u(\frac{x}{\theta^{\beta}})$ (so that $\|u_{\theta}\|_2=\theta \|u\|_2$), then we have the following scaling 
laws:
\begin{equation}\label{scalingA}
A(u_{\theta}):=\frac 12\int_{\R^{3}} |\nabla u_{\theta}|^{2}dx=\theta^{2-2\beta}A(u) 
\end{equation}
\begin{equation}\label{scalingN}
N(u_{\theta})=\frac {1}{4}\int_{\R^{3}} \phi_{u_{\theta}} |u_{\theta}|^{2}dx
=\theta^{4-\beta}N(u) 
\end{equation}
\begin{equation}\label{scalingM}
M(u_{\theta})=-\frac{1}{p} \int_{\R^{3}} |u_{\theta}|^{p}dx
=\theta^{(1-\frac{3}{2}\beta)p+3\beta}M(u)\,. 
\end{equation}
We get
\begin{equation*}
I(u_\theta)=\theta^2\left(I(u)+(\theta^{-2\beta}-1)A(u)+(\theta^{2-\beta}-1)N(u)+(\theta^{(1-\frac{3}{2}\beta)p+3\beta -2}-1)M(u)\right),
\end{equation*}
\begin{equation}\label{ff}
I(u_{\theta})=\theta^2\left(I(u)+f(\theta, u)\right)
\end{equation}
where 
\begin{equation}
f(\theta, u):=(\theta^{-2\beta}-1)A(u)+(\theta^{2-\beta}-1)N(u)+(\theta^{(1-\frac{3}{2}\beta)p+3\beta -2}-1)M(u). 
\end{equation}
We distinguish between the case $p=\frac{8}{3}$ and $3<p<\frac{10}{3}$.\\
\\
\textit{Case  $p=\frac{8}{3}$:}\\
\\
We notice that for $\beta=-2$ we get
$$I(u_{\theta})=\theta^6 A(u)+\theta^{6}N(u)+\theta^{4p-6}M(u)$$
and that $4p-6<6$ for $2<p<3$. Hence for $\theta\rightarrow 0 $ we have  $I(u_{\theta})\rightarrow 0^-$ which proves the first claim.\\
Let $u_n$ be a minimizing sequence in $B_{\mu}$ with $I_{\mu}<0$, then
\begin{eqnarray}
a_1<A(u_n)<a_2 \nonumber \\
m_1<M(u_n)<m_2. \nonumber
\end{eqnarray}
We get for $\beta=0$
$$f(\theta, u_n)=(\theta^{2}-1)N(u_n)+(\theta^{p-2}-2)M(u_n)
.$$
We have $\frac{d}{d \theta}f(\theta, u_n)|_{\theta=1}<0$ provided that
\begin{equation}\label{coo}
2N(u_n)+(2-\frac{8}{3})||u_n||_{L^{\frac{8}{3}}}^{\frac{8}{3}}<0.
\end{equation}
Relation \eqref{coo} holds for $\mu$ sufficently small 
recalling the following inequality
\begin{equation}
N(u_n)\leq C ||u_n||_2^{\frac{4}{3}}||u_n||_{L^{\frac{8}{3}}}^{\frac{8}{3}}.
\end{equation}
Indeed we have
$$2N(u_n)+(2-\frac{8}{3})||u_n||_{L^{\frac{8}{3}}}^{\frac{8}{3}}\leq ||u_n||_{L^{\frac{8}{3}}}^{\frac{8}{3}}(C\mu^{\frac{4}{3}}+(2-\frac{8}{3}))<0.$$
We notice that
$$\frac{d^2}{d \theta^2}f(\theta, u_n)=2N(u_n)-\frac{1}{p}(\frac{8}{3}-2)(\frac{8}{3}-3)\theta^{\frac{8}{3}-4}||u_n||_{L^{\frac{8}{3}}}^{\frac{8}{3}}>0,$$
and that
$$\frac{d}{d \theta}f(\theta, u_n)=0$$
for $\bar \theta$ fulfilling
$$\bar \theta^{-\frac{4}{3}}=2\frac{N(u_n)}{||u_n||_{L^{\frac{8}{3}}}^{\frac{8}{3}}}\leq C \mu^{\frac{4}{3}}$$
Thefore we find that for $\mu$ sufficently small $f(\theta, u_n)<0$ for
$\theta\in (1, \frac{C}{\mu})$.   We get 
$$I_{\theta \mu}< \theta^2 I(u_n)=\theta^2I_{\mu},$$
for $\theta \in (1,  \frac{C}{\mu})$.\\
Now  we argue as in \cite{L} observing that  for $\mu$ sufficently small 
\begin{multline*}
I_{\rho}=I_{\frac{\rho}{\mu}\mu}<\frac{\rho^2}{\mu^2}I_{\mu}=\frac{\rho^2-\mu^2+\mu^2}{\mu^2}I_{\mu}=\\\frac{\rho^2-\mu^2}{\mu^2}I_{\frac{\mu}{\sqrt{\rho^2-\mu^2}}\sqrt{\mu^2-\rho^2}}+I_{\mu}<I_{\mu}+I_{\sqrt{\rho^2-\mu^2}}.
\end{multline*}
\textit{Case  $3<p<\frac{10}{3}$:}\\
\\
We notice that for $\beta=-2$ we get
$$I(u_{\theta})=\theta^6 A(u)+\theta^{6}N(u)+\theta^{4p-6}M(u)$$
and that $4p-6>6$ for $3<p<\frac{10}{3}$. Hence for $\theta$ sufficently large we have $I(u_{\theta})<0$ which proves the first claim.\\
Let $u_n$ be a minimizing sequence in $B_{\mu}$ with $I_{\mu}<0$, then
\begin{eqnarray}
0<k_1<A(u_n)<k_2 \nonumber \\
0<\eta_1<|M(u_n)|<\eta_2. \nonumber
\end{eqnarray}
Indeed if $A(u_n)=o(1)$ we have $|M(u_n)|=o(1)$ and $I_{\mu}=0$.
For $\beta=-2$ we have
$$f(\theta, u_n):=(\theta^{4}-1)A(u)+(\theta^{4}-1)N(u)+(\theta^{4p-8}-1)M(u),$$
with $4p-8>4$ and 
\begin{equation}\label{asym2}
\frac{d}{d \theta}f(\theta, u_n)|_{\theta=1}<0,  \frac{d^2}{d \theta^2}f(\theta, u_n)<0 \text{ for all } \theta>1 .
\end{equation}
In order to show
\eqref{asym2} we have first
\begin{equation}\label{derivneg}
\frac{d}{d \theta}f(1, u_n)|_{\theta=1}=4 (A(u_n)+N(u_n))+(4p-8) M(u_n)<k<0.
\end{equation}
 Moreover we have
\begin{equation}\label{derivsecond}
 \frac{d^2}{d \theta^2}f(\theta, u_n)=12 \theta^{2}(A(u_n)+B(u_n))+(4p-8)(4p-9)\theta^{4p-10}M(u_n)<k<0.
\end{equation}
Thanks to \eqref{derivneg} and \eqref{derivsecond} we get 
$$f(\theta, u_n)<k(\theta)<0 \text{ for all } \theta>1, \ n\in \N,$$
and hence
$$I_{\theta \mu}< \theta^2 I(u_n)=\theta^2I_{\mu}.$$  
Let us suppose that $\mu<\sqrt{\rho^2-\mu^2}$. We distinguish three cases
\begin{itemize}
\item $\mu<\sqrt{\rho^2-\mu^2}<\rho_2$
\item $\mu<\rho_2<\sqrt{\rho^2-\mu^2}$
\item $\rho_2<\mu<\sqrt{\rho^2-\mu^2}$
\end{itemize}
The first case is trivial. For the second one
we have $I_{\sqrt{\rho^2-\mu^2}}>I_{\rho}$ and we conclude. For the third case we argue as for $p=\frac{8}{3}.$
\end{proof}

\begin{prop}\label{mma}
If $2<p<\frac{10}{3}$, then the functionals $N$
and $M$ fulfill \eqref{(2)}, \eqref{(3)}, \eqref{(5)}, \eqref{(6)}. 
\end{prop}
\begin{proof}
By Lemma \ref{stimap}, any minimizing sequence is bounded in the
$H^1$-norm. Hence $\set{u_n}$ is bounded in all $L^s$ norms for $s
\in [2,2^*]$ and there exists $\bar u \in H^{1}(\R^{3})$ such that $u_n
\rightharpoonup \bar u$ in $H^{1}(\R^{3}).$

 The functionals $M$ and $N$ satisfy the condition \eqref{(2)} (see \cite{BL} and Lemma 2.2 in 
 \cite{ZZ}) .

We have, by the convolution and Sobolev inequalities
$$
N(u_{n})=\int_{\R^{3}} \phi_{u_{n}}u_{n}^{2}dx\le C \|u_{n}\|_{12/5}^{4}\le C\|u_{n}\|_{2}^{3}\|\nabla u_{n}\|_{2}
$$
and than the relation \eqref{(3)}  follows from
\begin{eqnarray}
&&N(\alpha_n(u_n-\bar u))-N( u_n-u)=(\alpha_n^4-1)N(u_n -\bar u)=o(1) \nonumber  \\ 
&&M(\alpha_n(u_n-\bar u))-M( u_n-u)=(\alpha_n^p-1)M(u_n -\bar u)=o(1) \nonumber
\end{eqnarray}
since $\alpha_{n}\rightarrow1.$
Notice that thanks to the  classical interpolation inequality we have
\begin{equation*}
||u_n-u_m||_p\leq ||u_n-u_m||_2^{\alpha}||\nabla u_n-\nabla u_m||_2^{1-\alpha}\ \  \text{ where } \frac{\alpha}{2}+\frac{(1-\alpha)}{2^*}=\frac{1}{p}
\end{equation*}
and then on the minimizing sequence we get
\begin{equation*}
||u_n-u_m||_p=o(1).
\end{equation*}
We obtain, for $q=p/(p-1)$
\begin{equation*}
\int_{\R^{3}} |u_n|^{p-1}|u_n-u|dx \leq \left( \int_{\R^{3}} |u_{n}|^{q}\, dx \right)^{\frac 1 q}\ \left( \int_{\R^{3}} |u_{n}-u_{m}|^{p}\, dx \right)^{\frac 1 p}=o(1)
\end{equation*}
and then 
$$
\left| \int_{\R^{3}} (|u_{n}|^{p-1} - |u_{m}|^{p-1})(u_{n}-u_{m})\, dx \right| \le C\ \| u_{n}-u_{m} \|_{p}=o(1).
$$
This proves \eqref{(5)} for $M$.
The verification of  (\ref{(5)}) for $N$ follows from
\begin{eqnarray*}
\int_{\R^{3}}\phi_{u_{n}}u_{n}(u_{n}-u_{m})dx&\le&\|\phi_{u_{n}}\|_{6}\|u_{n}\|_{2}\|u_{n}-u_{m}\|_{3} \\
&\le&C\|{u_{n}}\|^{2}_{H^{1}(\R^{3})}\|u_{n}\|_{2}\|u_{n}-u_{m}\|_{3}=o(1)
\end{eqnarray*}
 Then condition (\ref{(6)}) is trivial.
\end{proof}

\medskip

Now we can conclude the proof of Theorem \ref{MT1}.
In case $p=\frac{8}{3}$ we can fix $\rho \in (0, \rho_1)$ due to the fact
that $I_{\rho}<0$ for all $\rho \in (0,\rho_1)$. In case $3<p<\frac{10}{3}$
we fix $\rho \in (\rho_2,+ \infty)$.\\
Let $\{u_n\}$ be a minimizing sequence  in $
 B_{\rho}$.
Notice also that for any sequence $y_n\in \R^n$ we have that
$u_n(.+y_n)$ is still a minimizing sequence for $I_{\rho}$.
This implies that the proof of the Theorem can be concluded 
provided that we show the existence of a sequence $y_n\in \R^3$ such that
the weak limit of $u_n(.+y_n)$ belongs to $B_\rho$ and that the convergence is strong in $H^1(\R^3)$. Notice that if 
\begin{equation*}\label{nonvani}
\lim_{n \rightarrow \infty} \left (\sup_{y\in \R^n} \int_{B(y,1)} |u_n|^2 dx \right )
=0
\end{equation*}
then $u_n\rightarrow 0$ in $L^{q}(\R^3)$ for any $q \in \left (2, 2^* \right )$, where $B(a,r)=\{x\in \R^{3}: |x-a|\le r\}$.
Since $I_{\rho}<0$ we have  that
\begin{equation*}\label{nelben}
\sup_{y\in \R^n} \int_{B(y,1)} |u_n|^2 dx \geq \mu>0.
\end{equation*}
In this case
we can choose $y_n\in \R^3$ such that
$$\int_{B(0,1)} |u_n (.+y_n)|^2 dx\geq \mu>0$$
and hence, due to the compactness of the embedding $H^1(B(0,1))\subset L^2(B(0,1))$,
we deduce that the weak limit of the sequence $u_n (.+y_n)$ is not the trivial function, so  $u_n\rightharpoonup\bar u \neq 0$.
Since the subadditivity condition holds,  we  can apply the abstract Lemma \ref{main-abs} and conclude the proof.

\section{The orbital stability}
In this section we prove Theorem \ref{stabil} following the ideas of \cite{CL}. First of all we recall
the definition of orbital stability.


We define
$$S_{\rho}=\{e^{i\theta }
u(x): \theta\in [0,2\pi), \|u\|_{2}=\rho, \ 
 I(u)=I_{\rho}\}.$$
We say that $S_{\rho}$ is {\sl orbitally stable} if
for every $\varepsilon>0$ there exists $\delta>0$ such that for any $\psi_{0}\in H^{1}(\R^{3})$ with
$\inf_{v\in S_{\rho}}\|v-\psi_{0}\|_{H^{1}(\R^{3};\C)}<\delta$ we have
$$\forall \, t>0 \ \ \ \inf_{v\in S_{\rho}
} \|\psi(t,.)-v\|_{H^{1}(\R^{3};\C)}<\varepsilon$$
where $\psi(t,.)$ is the solution of \eqref{SP} with initial datum $\psi_{0}$. We notice explicitly that $S_{\rho}$ is invariant by translation, i.e.
if $v\in S_{\rho}$ then also $v(.-y)\in S_{\rho}$ for any $y\in \R^{3}.$\\
We recall that the energy and the charge associated to $\psi(x,t)$ evolving according to \eqref{SP} are given by
\begin{eqnarray*}\label{energy}
E(\psi(x,t)):&=&\frac{1}{2}\int_{\R^3} |\nabla \psi|^2dx+\frac{1}{4}\int_{\R^3} (|x|^{-1}*|\psi|^{2}) |\psi|^2dx-\frac{1}{p}\int_{\R^3} |\psi|^pdx\\
&=&E(\psi(x,0))\nonumber
\end{eqnarray*}
and
\begin{equation*}\label{charge}
C(\psi(x,t)):=\frac{1}{2}\int_{\R^3} |\psi|^2dx=C(\psi(x,0)).
\end{equation*}
\\
So our action functional $I$ is exactly the energy.
In order to prove Theorem \ref{stabil} we argue by contradiction
assuming that there exists a $\rho$ such that $S_{\rho}$ is not orbitally stable. This means that there exists $\varepsilon>0$ and a sequence of initial
data $\{\psi_{n,0}\}\subset H^{1}(\R^{3})$ and $\{t_{n}\}\subset\R$ such that the maximal solution $\psi_{n}$, which is global
and $\psi_{n}(0,.)=\psi_{n,0}$, satisfies
\begin{equation*}
\lim_{n\rightarrow +\infty}\inf_{v\in S_{\rho}}\|\psi_{n,0}-v\|_{H^{1}(\R^{3})}=0 \ \ \ \text{ and }\ \ \inf_{v\in S_{\rho}}\|\psi_{n}(t_{n},.)-v\|_{H^{1}(\R^{3})}\ge\varepsilon
\end{equation*}
Then there exists $u_{\rho}\in H^{1}(\R^{3})$ minimizer of $I_{\rho}$  and $\theta\in \R$ such that $v=e^{i\theta}u_{\rho}$
and
$$\|\psi_{n,0}\|_{2}\rightarrow\|v\|_{2}=\rho\ \ \text{ and }\ \ \
I(\psi_{n,0})\rightarrow I(v)
=I_{\rho}
$$

Actually we can assume that
$\psi_{n,0}\in B_{\rho}$ (there exist $\alpha_{n}=\rho/\|\psi_{n,0}\|_{2}\rightarrow1$  so that
$\alpha_{n}\psi_{n,0}\in B_{\rho}$ and $ I(\alpha_{n}\psi_{n,0}) 
\rightarrow I_{\rho}$, i.e. we can replace $\psi_{n,0}$ with
 $\alpha_{n}\psi_{n,0}$). 
 
 So  $\{\psi_{n,0}\}$ is a minimizing sequence for $I_{\rho}$, and since
 $$I(\psi_{n}(.,t_{n}))=I(\psi_{n,0}),$$ 
also $\{\psi_{n}(.,t_{n})\}$ is a minimizing sequence for $I_{\rho}$. 
Since we have proved that every minimizing sequence has a subsequence converging (up to translation) in $H^{1}$-norm to a minimum on the sphere
$B_{\rho},$ 
we readily have a contradiction. 

Finally notice that, since in general, if $\psi(x,t)=|\psi(x,t)|e^{iS(x,t)}$ then $$I(\psi(x,t))=I(|\psi(x,t)|)+\int_{\R^{3}}|\psi(x,t)|^{2}|\nabla S(x,t)|^{2}dx,$$
we easily conclude that the minimizer $u_{\rho}$ has to be real valued.
 
\section{Application to a biharmonic Schr\"odinger equation}
In this final section we apply the above abstract result to the
following Schrodinger equation involving the biharmonic operator
\begin{equation}\label{NLSVQbi}
i\psi_{t} - \Delta^2 \psi  -F'(|\psi|) \frac{\psi}{|\psi|}=0,
\hbox{ } (t,x) \in \R\times \R^N \ \ \ N>4.
\end{equation}

The search of standing wave solution $\psi(x,t)=u(x)e^{-i \omega t}$ lead us to study the following semilinear equation
\begin{equation}\label{NLSVQbiell}
\Delta^2 u+ F(u)=-\omega u
\end{equation}
which will be studied by minimizing 
%
%
%
 the functional  $J:H^2(\R^N)\rightarrow \R$ given by
\begin{equation*}
J(u)=\frac{1}{2} \int_{\R^{N}} |\Delta u|^{2}dx+\int_{\R^N}F(u)dx
\end{equation*}
on $B_{\rho}=\{u\in H^{2}(\R^{N}):\|u\|_{2}=\rho\}$, namely studing the minimization problem
\begin{equation*}
J_{\rho}=\min_{B_{\rho}}J(u)
\end{equation*}
where $\omega$ is seen as the Lagrange multiplier.
\medskip

We make the following hypothesis on the nonlinearity
\begin{equation}\tag{$F_{p}$}\label{Fp}
\begin{array}{ll}
|F^{\prime }(s)|\leq c_{1}|s|^{q}+c_{2}|s|^{p}& \text{ for some
} \ \ 2<q\leq p<\frac{N+4}{N-4}  
\end{array}
\end{equation}
\begin{equation}\tag{$F_{0}$}\label{F0}
F(s)\geq -c_1s^{2} -c_2s^{2+\frac{4}{N}}\ \ \text{ with }\ \ c_1, c_2> 0 
\end{equation}
\begin{equation}
\exists s_{0}\in (0, +\infty) \ \ \text{ such that }\ \ F(s_{0})<0.
\tag{$F_{1}$} \label{F1}
\end{equation}
So we get the following result.
\begin{thm}\label{minNLSapp}
Let \eqref{Fp}, \eqref{F0} and \eqref{F1} hold. Then there exists $\rho_0$ such that for all $\rho>\rho_0$,
$J_{\rho}$ is achieved on $u_{\rho}$ and $(\omega_{\rho}, u_{\rho})$ is a solution of \eqref{NLSVQbiell}.
\end{thm}
In order to apply the astract Lemma \ref{main-abs} to the functional
\begin{equation*}\label{eqbilap}
J(u)=\frac{1}{2}||u||_{D^{2,2}}^2+T(u) \ \ \ \text{ where }\  \ T(u)=\int_{\R^{N}} F(u)dx
\end{equation*}
we need to prove the boundedness of $J$ on $B_{\rho}$ and the subadditivity condition.
\begin{prop}\label{boundedness}
If \eqref{Fp}, \eqref{F0} and \eqref{F1} hold, then there exists $\rho_0$ such that for all $\rho>\rho_0$ 
\begin{itemize}
\item $-\infty<J_{\rho}<0$;
\item any
minimizing sequence $\{u_n\}\subset  B_{\rho}$ for $J$
is bounded in $H^2(\R^N)$.
\end{itemize}
\end{prop}
\begin{proof}
By arguing as in \cite{BV}  we have that $J_{\rho}>-\infty$ and that the functional is coercive.
We build a sequence of  radial functions $\{u_n\}$ in $H^{2}(\R^N)$  such
that
$J(u_n)<0$ for large $n$.
The sequence is defined as follows:
\begin{equation*}
u_n(r)=
\left\{
\begin{array}{lll}
s_0&&r<R_n;\\
s_0\cos^2(\frac{\pi}{2}(r-R_n))&&R_n\leq r\leq R_n+1;\\
0&& r>R_n+1.
\end{array}
\right.
\end{equation*}

We show that
$J(u_n)<0$ when $R_n \rightarrow +\infty$. Notice that for a radial function
$u$ the laplacian is given by
\begin{equation*}\label{proplapl}
\Delta u=\frac{\partial^2 u}{\partial r^2}+\frac{\partial u}{\partial r}
\left(\frac{N-1}{r}\right).
\end{equation*}
After some computation we have
\begin{equation*}
\frac{\partial u_n(r)}{\partial r}=
\left\{
\begin{array}{lll}
0&&r<R_n;\\
-\pi s_0 \cos(\frac{\pi}{2}(r-R_n)\sin(\frac{\pi}{2}(r-R_n))&&R_n\leq r\leq R_n+1;\\
0&& r>R_n+1
\end{array}
\right.
\end{equation*}
and
\begin{equation*}
\frac{\partial^2 u_n(r)}{\partial r^2}=
\left\{
\begin{array}{lll}
0&&r<R_n;\\
\frac{\pi^2}{2}s_0 \left(\sin^2(\frac{\pi}{2}(r-R_n))-\cos^2(\frac{\pi}{2}(r-R_n))\right)&&R_n\leq r\leq R_n+1;\\
0&& r>R_n+1.
\end{array}
\right.
\end{equation*}
Then we get 
\begin{eqnarray*}
J(u_n)&=& \int_{\R^N} \frac12 |\Delta u_n|^2+F(u_n)dx \leq \\
&\leq  &C_1 \int_{R_n}^{R_n+1}\left[
\left|\frac{\partial^2 u_n(r)}{\partial r^2}+\frac{\partial u_n(r)}
{\partial r}(\frac{N-1}{r})\right|^2+
\sup_{|s|<s_0} F(s) \right]
r^{N-1}dr +\\
&+& C_2\int_0^{R_n}F(s_0)r^{N-1}dr
\end{eqnarray*}
where $C_1$ and $C_2$ are strictly positive constants. We have $F(s_0)<0$ and, thus, an easy growth estimate gives $J(u_n)<0$ for $R_n\rightarrow +\infty$.
\end{proof}
\begin{prop}\label{subadditivity}
For any $\rho>\rho_0$ and $0<\mu<\rho$ the following subadditivity condition holds
\begin{equation}\label{subbbilap}
J_{\rho}<J_{\mu}+J_{\sqrt{\rho^2-\mu^2}}.
\end{equation}
\end{prop}
\begin{proof}
Let us define $u_{\lambda}(x)=u(\frac{x}{\lambda^{\frac{2}{N}}})$ (so that $\|u_{\lambda}\|_{2}=\lambda \|u\|_{2}$).
We have
\begin{equation*}
J(u_\lambda)=\frac{\lambda^{2-\frac{4}{N}}}{2}||u||_{D^{2,2}}^2+\lambda^2T(u).
\end{equation*}
and then
$$J_{\lambda \mu}\leq \lambda^{2}\left( \frac{1}{2}||u_n||_{D^{2,2}}^2+T(u_n)\right)+\frac{1}{2}\left( \lambda^{2-\frac{4}{N}}-\lambda^2\right)||u_n||_{D^{2,2}}^2$$
for any minimizing sequence $\{u_n\}\subset B_{\mu}$. Taken $\mu$ such that $J_{\mu}<0$ and $\lambda>1$ we obtain
$$J_{\lambda \mu}< \lambda^2 J(\mu).$$
By arguing as in Lemma \ref{fondamentale} we have \eqref{subbbilap}.
\end{proof}
\begin{prop}
If \eqref{Fp}, \eqref{F0} and \eqref{F1} hold then the functional $T$ fulfills \eqref{(2)}, \eqref{(3)}, \eqref{(5)}, \eqref{(6)}
\end{prop}
\begin{proof}
It follows as in Proposition \ref{mma}. Condition \eqref{(2)} follows from
standard arguments.
\end{proof}

\begin{proof}[Proof of Theorem \ref{minNLSapp}]
We argue as in the proof of Theorem \ref{MT1}. Recall that
if $\{u_n\}$ is a bounded sequence in $H^2(\R^N)$  such that
\begin{equation*}
\lim_{n\rightarrow \infty} \left (\sup_{y\in \R^n} \int_{B(0,1)} |u_n|^2 dx \right )
=0
\end{equation*}
then $u_n\rightarrow 0$ in $L^{q}(\R^n)$ for any $q \in \left (2, \frac{2N}{N-4}\right )$. 
The proof of this fact is given in \cite{BV}. Finally we apply Lemma \ref{main-abs} to the functional $J(u)$.
\end{proof}
Finally the orbital stability of the standing waves is proved. As in the previous section,
we define
$$S_{\rho}=\{e^{i\theta }
u(x): \theta\in [0,2\pi), \|u\|_{2}=\rho, \ 
 J(u)=J_{\rho}\}$$
and we say that $S_{\rho}$ is {\sl orbitally stable} if
for every $\varepsilon>0$ there exists $\delta>0$ such that for any $\psi_{0}\in H^{2}(\R^{N})$ with
$\inf_{v\in S_{\rho}}\|v-\psi_{0}\|_{H^{2}(\R^{N})}<\delta$ we have
$$\forall \, t>0 \ \ \ \inf_{v\in S_{\rho}
} \|\psi(.,t)-v\|_{H^{2}(\R^{N})}<\varepsilon$$
where $\psi(.,t)$ is the solution of \eqref{NLSVQbi} with initial datum $\psi_{0}$.
Arguing as for the Schrodinger-Poisson equation we obtain the following
\begin{cor}
Let \eqref{Fp}, \eqref{Fp} and \eqref{F1} hold. Then there exists $\rho_0>0$ such that
for any $\rho\in (\rho_0,+ \infty)$   the standing waves $\psi_{\rho}(t,x)=e^{-i\omega_{\rho}t}u_{\rho}(x)$ are orbitally stable solutions of  \eqref{NLSVQbi}.
\end{cor}


\begin{thebibliography}{}

\bibitem {AzzPo}{A. Azzollini, A. Pomponio, }{Ground state solutions for
the nonlinear Schr{\"o}dinger-Maxwell equations, } {J. Math. Anal. Appl. {\ 345}
(2008), no. 1, 90--108.}

\bibitem{BB}{J. Bellazzini, C. Bonanno}, {Nonlinear Schr\"odinger equations with strongly singular potentials}, Proc. Roy. Soc. Edimburgh sec A (in press), arXiv:0903.3301

\bibitem{BV} J.Bellazzini, N.Visciglia, On the orbital stability
for a class of nonautonmous NLS, Indiana Univ. Math. J., online
preprint

\bibitem {BF}{V. Benci, D. Fortunato, }{An eigenvalue problem for the
Schr{\"o}dinger-Maxwell equations, }{Topol. Methods Nonlinear Anal. {\ 11} (1998),
283--293.}
%

\bibitem{BL} H. Brezis, E. Lieb, A relation between pointwise convergence of functions and convergence of functionals.  Proc. Amer. Math. Soc.  88  (1983), no. 3, 486--490.

\bibitem{C}{T. Cazenave, }{Semilinear Schr\"odinger equation, }{ Courant Lecture Notes in Mathematics, 10. New York: New York University, Courant Institute of Mathematical Sciences.}


\bibitem{CL} {T. Cazenave, P.L. Lions,}
{\ Orbital Stability of Standing Waves for Some Non linear
Schr\"odinger Equations}, Commun. Math. Phys.  85, 549-561
(1982).

\bibitem {Coc}{G.M. Coclite, }{\ A multiplicity result for the nonlinear
Schr{\"o}dinger-Maxwell equations, }{\ Commun. Appl. Anal. {\ 7} (2003),
417--423.}

\bibitem {DApMu}{T. D'Aprile, D. Mugnai, }{\ Solitary waves for nonlinear
Klein-Gordon-Maxwell and Schr{\"o}dinger-Maxwell equations, }{Proc. Royal Soc.
Edinburgh {134 A} (2004), 893--906.}

\bibitem {dAv}{P. d'Avenia, }{\ Non-radially symmetric solutions of nonlinear
Schr{\"o}dinger equation coupled with Maxwell equations, }{Adv. Nonl. Studies
 2 (2002), 177--192.}
\bibitem{GSS}{M. Grillakis, J. Shatah, W. Strauss, }{\ Stability theory of solitary waves in the presence of symmetry. I. }
{ J. Funct. Anal. 74 (1987), no. 1, 160--197.}

\bibitem{IL}{I. Ianni, S. Le Coz, }{\ Orbital stability of standing waves of a semiclassical nonlinear Schršdinger-Poisson equation, }
{Adv. Differential Equations 14 (2009), no. 7-8, 717--748.}

\bibitem {Kik}{H. Kikuchi, }{\ On the existence of solutions for a elliptic
system related to the Maxwell-Schr{\"o}dinger equations, }{Nonl. Anal. {\ 67}
(2007), 1445--1456.}

\bibitem {K}{H. Kikuchi, }{\ Existence and stability of standing waves for Schršdinger-Poisson-Slater equation, }
{ Adv. Nonlinear Stud. 7 (2007), no. 3, 403--437.}

\bibitem{L} {P. L. Lions, }{\ The concentration-compactness principle in the Calculus of Variation. The locally compact case,
part I and II, }{ Ann. Inst. H. PoincarŽ Anal. Non LinŽaire 1 (1984), 109--145 and 223--283.}

\bibitem{PS}{L. Pisani, G. Siciliano, }{Neumann condition in the Schr\"odinger - Maxwell system, }{ Topol. Methods Nonlinear Anal.{\ 29} (2007), 251--264.}
\bibitem {jfa}{D. Ruiz, }{\ The Schr{\"o}dinger-Poisson equation under the
effect of a nonlinear local term, }{J. Funct. Anal. {\ 237} (2006), 655--674.}

\bibitem{SS} O. Sanchez, J. Soler, {\ Long time dynamics of the  Schr\"odinger-Poisson-Slater system}, {Journal of Statistical Physics 114 (2004), 179--204.}

\bibitem {WZh}{Z. Wang, H.S. Zhou, }{\ Positive solution for a nonlinear
stationary Schr\"{o}dinger-Poisson system in $\mathbb{R}^{3}$, }{Discrete Contin. Dyn. Syst.
 {\ 18} (2007), 809--816.}
 
\bibitem{ZZ}{L. Zhao, F. Zhao, }{On the existence of solutions for the Schr\"odingerÐPoisson equations, }{J. Math. Anal. Appl. {346} (2008) 155--169.}
\end{thebibliography}
\end{document}